\def\Q{{\bf Q}}\def\P{{\bf P}}

\documentclass[11pt]{amsart}
\usepackage{amsmath,amssymb}
\newcommand{\calB}{\mathcal{B}}
\newcommand{\N}{\mathbb{N}}
\newcommand{\Z}{\mathbb{Z}}

\newtheorem{thm}[section]{Theorem}
\newtheorem{cor}[section]{Corollary}

\newtheorem{defn}[section]{Definition}

\begin{document}
\title{Independence and Alpern Multitowers}
\author{James T. Campbell}
\email{jcampbll@memphis.edu}
\author{Randall McCutcheon}
\email{rmcctchn@memphis.edu}
\author{Alistair Windsor}
\email{awindsor@memphis.edu}
\address{Department of Mathematical Sciences\\ Memphis\\ TN 38152\\ U.S.A.}

\maketitle

\begin{abstract}
	Let $T$ be any invertible, ergodic, aperiodic measure-preserving transformation of a Lebesgue probability space $(X, \calB, \mu)$, and \P\, any finite measurable partition of $X$.  We show that a (finite) Alpern multitower may always be constructed whose base is independent of \P.  \bigskip
	
	\textit{2010 Mathematics Subject Classification}: 28D05, 37M25, 60A10 \medskip
	
	\textit{Keywords:} Rohlin tower, Alpern tower, independent sets, measure-preserving transformation, probability space. 
\end{abstract}

There are two well-known theorems in measurable ergodic theory which provide stronger versions of the classical Kakutani-Rokhlin tower construction: 

\begin{enumerate}
	\item[(1)] In any suitable system, given a finite measurable partition of the underlying space, there exists a Kakutani-Rokhlin tower whose levels are independent of the partition. 
	\item[(2)] In any suitable system, there exists an Alpern tower (see the paragraph just below the statement of Theorem \ref{thm:Main}). 
\end{enumerate}

The result described in (1) was used most famously in the proof by Ornstein of his Isomorphism Theorem for Bernoulli shifts (\cite{Or:70}), and has proven useful to others over time (see \cite{Ko:04} for a broader discussion). The Alpern tower result has also been extensively applied (\cite{EP:97} has a brief overview, see \cite{Ka:12} for a more recent application), and has been shown to be an equivalent form of two other important results in ergodic theory (\cite{AP:08}). 

Our colleague Steven Kalikow expressed the thought that it would be reasonable to hope that someone could combine the two and get even stronger results. In this short note we accomplish the first of these goals by proving the following result: 

\begin{thm}\label{thm:Main}
  Given an invertible ergodic aperiodic measure-preserving transformation $T$ of a
  Lebesgue probability space $(X, \calB, \mu)$, a finite measurable partition
  $\P=\{P_1,P_2,\ldots ,P_r\}$ of $X$, relatively prime
  $n_1, \ldots ,n_t\in \N$, and positive real numbers
  $x_1, \ldots , x_t$ satisfying $\sum_{i=1}^t x_i n_i = 1$, there
  exist measurable sets $B_i$, $1\leq i\leq n$, such that
  \begin{enumerate} \item $\mu(B_i)=x_i$, \item
    $\{ T^j B_i :1\leq i\leq t, 0\leq j<n_i\}$ partitions $X$, and
  \item \label{part:ind} $B_i$ is independent of $\P$, $1\leq i\leq t$. $_\blacksquare$ \end{enumerate}

\end{thm}

The novelty here is (\ref{part:ind}),  independence of the base from the given partition \P. The main result by Alpern in \cite{A:79} (Corollary 2) is that under the hypothesis of Theorem \ref{thm:Main} (sans \P, which plays no role in Alpern's theorem), conclusions (1) and (2) hold. In fact he does not require $T$ to be ergodic, only aperiodic. It is fairly clear (using the ergodic decomposition) that our result will also hold for aperiodic invertible m.p.t.'s, but the technical steps to show this (regarding questions of measurability) would dilute the efficiency of this note, so we do not include them. 

 The collection $\{ T^j B_i :1\leq i\leq t, 0\leq j<n_i\}$ is what we earlier referred to as an {\em Alpern tower}.  Alpern also proved (\cite{A:81}) a denumerable version of his tower theorem, namely, that the collection $\{B_i\}$ in Theorem \ref{thm:Main} may be denumerable,  as long as the heights $\{n_i\}$ form a relatively prime set. This result is known as {\bf MRT}, the Multiple Rokhlin Tower decomposition. Sahin (\cite{S:09}) proved a generalization of MRT for $\mathbb Z^d$ actions.

At this point we do not know if Theorem \ref{thm:Main} can be extended to either of these more general settings. \smallskip

Theorem \ref{thm:Main} will be derived from the following special case, proved in \cite{CCKKM:15}:

\begin{thm}\label{thm:SpecialCase} Let $(X, \calB, \mu)$ be a Lebesgue
  probability space and
  $\P$ a finite measurable partition of $X$. For any ergodic aperiodic
  invertible measure-preserving transformation $T$ of $X$ and
  $N \in \N$, there exists a Rokhlin tower of height $N$ with base $B$
  and error set $E$ (i.e. $\{ B, TB, T^2B, \ldots ,T^{N-1}B, E\}$
  partitions $X$) with $T(E) \subset B$ and $B$ independent of the
  partition $\P$. $_\blacksquare$ \end{thm}

The special case of a Rokhlin tower of height $N$ with base $B$
and error set $E$  with $T(E) \subset B$ is often referred to as an {\em Alpern tower of height $N$}. There is another way to visualize such a tower which will be helpful in understanding the proof of Theorem \ref{thm:Main}, as follows. 

\begin{defn} \label{def:1}
	By a {\em tower over B} we will mean a set $B \subset X$, called the {\em base}, and a countable partition $B = B_1 \cup B_2 \cup \cdots$, together with their images  $T^iB_j$, $0 \le i < j$, such that the family $\{T^iB_j : 0 \le i < j\}$ consists of pairwise disjoint sets. If this family partitions $X$, we will say that the tower is {\em exhaustive}. 
\end{defn}

Some of the $B_i$ could be empty, in which case we discard them. If a tower over $B$ is exhaustive and $B = B_N \cup B_{N+1}$, then 
$\{B, TB, \ldots, T^{N-1}B, E=T^N B_{N+1}\}$ partitions $X$, and it must be the case (by invertiblity) that $T(E) \subset B$. Hence, one may visualize an Alpern tower of height $N$ as two columns, one sitting over $B_N$ of height $N$, one sitting over $B_{N+1}$ of height $N+1$, with $E = T^{N}(B_{N+1})$. Since $T$ is measure-preserving and invertible, and $T(E) \subset B$, we must have $T^{-1}E = T^{N-1}B_{N+1}$, $T^{-2}E = T^{N-2}B_{N+1}$, etc.. This structure is used explicitly in the proof of Theorem \ref{thm:Main}, in particular at Observation 1.

It was remarked in \cite{CCKKM:15} that the proof of Theorem \ref{thm:SpecialCase}
shows that the error set $E$ can be taken as small as desired. We give
a short proof here as part of the following Corollary, which also shows that in fact much more than the base of the tower may be taken to be independent of the partition. 

\begin{cor}\label{cor} In the conclusion of Theorem \ref{thm:SpecialCase} one may require
  that $\mu(E)$ is as small as desired and that all of the sets 
  \begin{enumerate}
  \item $T^j E$,  for  $-N\leq j\leq 0$, and 
  \item $T^j B$, for $0\leq j<N$
  \end{enumerate}
  are independent of the partition $\P$. 
\end{cor}

\begin{proof}[Proof of Corollary \ref{cor}]
  We first prove that the two families of sets may be taken to be
  independent of the partition $\P$. Let
  $\Q=\bigvee_{i=-N+1}^N T^i \P$ and let, by Theorem
  \ref{thm:SpecialCase}, $B$ be the base of a Rokhlin tower of height
  $N$ with an error set $E$ satisfying $TE\subset B$ and  $B$ is
  independent of $\Q$. For $0\leq j<N$ we have $T^jB$ is independent
  of $T^j \Q$ and, hence, of the coarser partition
  $\bigvee_{i=0}^N T^i \P$. Since
  $\{ B, TB, T^2B, \ldots ,T^{N-1}B, E\}$ partitions $X$ this implies
  that $E$ is independent of $\bigvee_{i=0}^N T^i \P$. In particular,
  for $-N\leq j\leq 0$ we have that $T^{j} E$ is independent of the
  partition $\P$ .

  Now we prove that for any $\epsilon>0$ we may take $\mu(E) <
  \epsilon$ in the conclusion of Theorem \ref{thm:SpecialCase}. Chose
  $k \in \N$ such that $\frac{1}{k N} < \epsilon$. Now we construct a tower of height $k N$ with base
  $B'$ such that, for $0\leq j<kN$, $T^j B'$  is independent of $\P$. We are given $T(E) \subset B'$ so in particular $\mu(E) <
  \frac{1}{k N}$. Now define 
  \begin{equation*}
    B = \cup_{i=0}^{k-1} T^{i N} B'.  
  \end{equation*}
  By construction $B$ is again independent of $\P$.

 \end{proof}

We can now use Corollary \ref{cor} to prove Theorem \ref{thm:Main}.

 \begin{proof}[Proof of Theorem \ref{thm:Main}]
   Since $n_1, \dots, n_t$ are relatively prime we may choose
   $z_1,\ldots ,z_t \in \Z$ such that
    $\sum_i z_in_i = 1$. Choose $L >  \max\{|z_i|: 1\leq i\leq t\}$
    such that $L$ is
    divisible by each $n_i$ and let $N= L(n_1+n_2+\cdots +n_t)$.
    Then
    \begin{equation*}
      N+1 = \sum_{i=1}^t n_i\big(z_i+L\big) 
    \end{equation*}
    where each $z_i+L \in \N$ by our choice of $L$.

    Choose by Corollary \ref{cor} a measurable base $B$ and error set $E$, with
    \begin{equation} \label{eq:1}
      \mu(E)< \frac{\min \{x_i:1\leq i\leq t\}}{N+1}
    \end{equation}
    such that $\{ B, TB, T^2B,$ $\ldots ,T^{N-1}B, E\}$ partitions
    $X$, $TE\subset B$ and such that both
    \begin{enumerate}
    \item $T^j E$ for  $-N\leq j\leq 0$, and 
    \item $T^j B$, for $0\leq j<N$,
    \end{enumerate}
    are independent of the partition $\P$. \medskip
    
    \noindent {\bf Observation 1:}  As a consequence of $TE \subset B$ we have $T^{-1} E \subset
    T^{N-1}B$ and as a result $T^{-N} E \subset B$. We may thus
    partition $B = B' \cup B''$ where $B' = T^{-N} E$ and $B'' = B
    \setminus B'$. There is a tower of height $N+1$ over $B'$ 
    \begin{equation*}
      B' = T^{-N}E, \quad T B' =T^{-N+1}E, \quad\dots\quad, T^{N-1} B'
      =T^{-1} E, \quad
      T^N B' = E
    \end{equation*}
    and from Corollary \ref{cor} each
    rung of the tower is independent of $P$. Similarly, there is a
    tower of height $N$ over $B''$
    \begin{equation*}
      B'' = B\setminus B',\quad TB'' = T B\setminus T B',\quad \dots
      \quad, T^{N-1} B'' =  T^{N-1} B\setminus T^{N-1} B'
    \end{equation*}
    and, since from Corollary \ref{cor} both $ T^j B $ and $T^j B'$
    are independent of $\P$, we have that each rung of the tower is independent of
    $\P$. 

    We partition (working our way up from the bottom) the tower of
    height $N+1$ over $B'$ so that for each $1\leq i \leq t$ there
    are $(z_i+L)$ towers  of height $n_i$ each consisting of rungs of
      the original tower. Let $B_i'$ be the union of the bases of the
      $z_i + L$ towers of height $n_i$.

      The idea here is that each $B_i'$ will be a portion of the $B_i$
      we seek; we remark that because each rung of the tower of height
      $N+1$ over $B'$ is independent of $\P$, and each $B_i'$ is a
      union of some of these, each set $B_i'$ is independent of $\P$.
      Moreover, $\{ T^j B_i' :1\leq i\leq t, 0\leq j<n_i\}$ partitions
      the tower of height $N+1$ over $B'$ and
    \begin{equation*}
      \mu(B_i')= (z_i+L)\mu(E)< x_i, \;\; 1\leq i\leq t, 
    \end{equation*}
by (\ref{eq:1}).

    We next turn to the tower of height $N$ over $B''$. For $ w
    \in \P^N$, a word of length $N$ whose letters are cells of $\P$,  define 
    \begin{equation*}
      C^{(w)} = \{ x \in B'' : T^{j-1} x \in w_j \text{ for $ 1 \leq
        j \le N$} \}. 
    \end{equation*}
 For each $w \in \P^N$  further partition
    $C^{(w)}$ into pieces $C^{(w)}_i$, $1\leq i\leq t$, such that
    $\mu(C^{(w)}_i) = r_i \mu(C^{(w)})$, where
    \begin{equation*}
      r_i = { n_i\big(x_i-\mu(B_i')\big)  \over N\mu(B'')}.
    \end{equation*}
    One notes that the $r_i$ sum to 1. \medskip

    Let $C_i = \bigcup_{w \in \P^N} C^{(w)}_i$. We have $C_1, \dots,
    C_t$
    partitions $B''$. Since every rung of the tower of height $N$ over
    $B''$ is independent of $\P$ and $C_i$ contains an equal
    proportion of every $C^{(w)}$, every rung of the $N$ tower over
    $C_i$ is independent of $\P$ as well.

    Next partition, again
    working from the bottom up, the tower over $C_i$ of height $N$
    into $N/n_i$
    towers of height $n_i$, $1\leq i\leq t$. Let $B_i''$ be the
    union of the bases of these towers of height $n_i$. Then
    $\{ T^j B_i'' :1\leq i\leq t, 0\leq j<n_i\}$ partitions the
    tower over $B''$ of height $N$, so if we let $B_i=B_i'\cup B_i''$ then, since as
    noted earlier $\{ T^j B_i' :1\leq i\leq t, 0\leq j<n_i\}$
    partitions the $(N+1)$ tower over $B'$, (ii) is satisfied.
    Moreover, each $B_i''$ is a union of rungs independent of $\P$,
    and hence is itself independent of $\P$. Therefore, since as noted
    earlier $B_j'$ is independent of $\P$, (iii) is satisfied. Finally
    \begin{equation*}
      \mu(B_i'') = {N \mu(C_i)\over n_i} ={Nr_i\mu(B_i'')\over n_i} =
      x_i-\mu(B_i')
\end{equation*}
so that 
\begin{equation*}
  \mu(B_i)=\mu(B_i')+\mu(B_i'') = x_i
\end{equation*}
as required.

\end{proof}

\newcommand{\etalchar}[1]{$^{#1}$}
\providecommand{\bysame}{\leavevmode\hbox to3em{\hrulefill}\thinspace}
\providecommand{\MR}{\relax\ifhmode\unskip\space\fi MR }
\providecommand{\MRhref}[2]{%
  \href{http://www.ams.org/mathscinet-getitem?mr=#1}{#2}
}
\providecommand{\href}[2]{#2}

\end{document}